\documentclass{amsart}
\usepackage{amsthm,amsmath,amssymb}
\usepackage{mathrsfs}
\DeclareMathAlphabet{\mathpzc}{OT1}{pzc}{m}{it}
\usepackage[numbers]{natbib}  
\usepackage{amsfonts,pgf}

\newtheorem{theorem}{Theorem}[section]
\newtheorem{lemma}[theorem]{Lemma}

\newtheorem{claim}[theorem]{Claim}

\newtheorem{corollary}[theorem]{Corollary} 
\theoremstyle{definition}
\newtheorem{definition}[theorem]{Definition}

\newtheorem{notation}[theorem]{Notation} 
\theoremstyle{remark}
\newtheorem{remark}[theorem]{Remark}

\newtheorem{obs}{Observation}[section]
\newtheorem{convention}{Convention}[section]
\newtheorem{proposition}{Proposition}

\newcommand{\M}{\mathcal{M}}
\newcommand{\I}{\mathcal{I}}

\newcommand{\J}{\mathcal{J}}
\newcommand{\In}{\mathbf{I}}
\newcommand{\Jn}{\mathbf{J}}

\newcommand{\Tn}{\mathbf{T}}
\newcommand{\Un}{\mathbf{U}}
\newcommand{\Vn}{\mathbf{V}}
\newcommand{\Wn}{\mathbf{W}}

\newcommand{\Mm}{\mathbb{M}}

\newcommand{\N}{\mathcal{N}}

\newcommand{\K}{\mathcal{K}}

\newcommand{\Uu}{\mathcal{U}}

\newcommand{\Ll}{\mathcal{L}}
\newcommand{\Ss}{\mathcal{S}}

\newcommand{\la}{\left <}
\newcommand{\ra}{\right >}
\newcommand{\uphp}{\upharpoonright}
\newcommand{\ov}{\overline}
\newcommand{\ar}{\rightarrow}
\newcommand{\qt}{\textrm{qftp}}

\newcommand{\ms}{\vspace{.05in}}
\newcommand{\bs}{\vspace{.1in}}
\newcommand{\lx}{<_{\textrm{lex}}}
\newcommand{\nin}{\noindent}

\newcommand{\cl}{\overline{\textrm{cl}}}

\newcommand{\W}{\omega^{<\omega}}

  \title{Indiscernibles, EM-types, and Ramsey Classes of Trees}
\author{Lynn Scow}

\begin{document}

\maketitle

\begin{abstract}
 It was shown in \cite{sc12}  that for a certain class of structures $\I$, $\I$-indexed indiscernible sets have the modeling property just in case the age of $\I$ is a Ramsey class. We expand this known class of structures from ordered structures in a finite relational language to ordered, locally finite structures which isolate quantifier-free types by way of quantifier-free formulas.  As a corollary, we obtain a new Ramsey class of finite trees.
\end{abstract}

\footnotemark{It was recently brought to my attention that the result in Corollary \ref{314} appears in \cite{le73} and was later surveyed in \cite{garo74}.  Subsequent proofs have been found by M. Sokic in \cite{so13} and by S. Solecki.

\section{Introduction}\label{1}

A generalized indiscernible set (which we will abbreviate as an \emph{indiscernible}) is a set of tuples from a model $\M$, $(a_i : i \in \I)$, indexed by a structure $\I$ in a homogeneous way: the complete type of a finite tuple of parameters $(a_{i_1}, \ldots, a_{i_n})$ in $\M$ is fully determined by the quantifier-free type of the indices $(i_1, \ldots, i_n)$ in $\I$.  If $\I$ is known, we call the indiscernible an \emph{$\I$-indexed indiscernible set.}  Generalized indiscernible sets were originally developed in \cite{sh90} and have been used in many places: \cite{bash12,lassh03,dzsh04,gu12,kim11,tats12}.  In \cite{dzsh04}, indiscernibles indexed by trees were studied, and a specific property was proved of them.  One of the main goals of \cite{sc12} was to consider this specific property generalized from a tree to an arbitrary structure, $\I$, named the \emph{modeling property (for $\I$-indexed indiscernibles)}, and relate this property to a combinatorial property of the age of $\I$.  The appropriate notion turned out to be the one of \emph{Ramsey class} (see Definition \ref{44}.)  A ``dictionary'' theorem was proved: if $\I$ is a structure in a finite relational language, linearly ordered by one of its relations, then the age of $\I$  is a Ramsey class just in case $\I$-indexed indiscernible sets have the modeling property (see Definition \ref{32}.)  It was conjectured that results might travel both ways through this dictionary: known Ramsey classes would yield new structures to index indiscernibles; known results on indiscernibles would yield new Ramsey classes.  In fact this is the case.
In Theorem \ref{37}, we extend this dictionary to the case where $\I$ is locally finite, linearly ordered by one of its relations, and has a certain technical property, \texttt{qfi}: quantifier-free types realized in $\I$ are isolated by quantifier-free formulas.  This generalizes the dictionary theorem to certain situations where we have an infinite language containing function symbols, in particular to the case where $\I$ is ordered and locally finite in a finite language. The \textit{locally finite-linearly ordered-\texttt{qfi}} case encompasses two indexing structures $\I$ from the literature, $I_0 = (\W,\unlhd,\wedge,\lx)$ and $I_s = (\W,\unlhd,\wedge,\lx,(P_n)_{n<\omega})$, where $\unlhd$, $\wedge$, $\lx$, $P_n$ are interpreted as the partial tree-order, the meet function in this order, the lexicographic order on sequences, and the $n$-th level of the tree, respectively.  It is known from \cite{kks11,tats12} that both of these structures index indiscernibles with the modeling property. Corollaries \ref{314} and \ref{315} conclude that the ages of $I_0$, $I_s$, respectively, form Ramsey classes.  The former constitutes an alternative proof of a known result (see \cite{nvt10,fou99}); the latter introduces a new example of a Ramsey class of finite trees.

In Section \ref{2} we give the basic lemmas around \texttt{qfi} and further develop a notion of EM-type used in \cite{kks11}.  In the process, we give restatements of certain definitions from \cite{sc12} in Definitions \ref{201}, \ref{31}, and \ref{555} that drop reference to a linear order on $\I$.  The technology of EM-types primarily addresses the question, ``what uniform definable character of an initial, indexed set of parameters may be preserved in an indiscernible indexed by the same set?''  In the technology developed in this section, there is no use of a linear order on the index structure, $\I$.  Though indiscernibles indexed by unordered $\I$ do not exist in \emph{all} structures $M$, (\cite{sc12}) the technical lemmas of this section are still of some independent interest for studying unordered indiscernibles in a limited setting. 

In Section \ref{3} we prove the main theorem, Theorem \ref{37}, that in the more general case of \textit{locally finite-linearly ordered-\texttt{qfi}}, $I$-indexed indiscernibles have the modeling property just in case $\textrm{age}(I)$ is a Ramsey class.  From this theorem we deduce the new partition result, Corollary \ref{314}, that $\textrm{age}(I_0)$ is a Ramsey class.

In Section \ref{4} we provide an alternate proof of the result that $I_0$-indexed indiscernibles have the modeling property (from \cite{tats12}) using only a result of \cite{fou99}, Theorem \ref{37}, and the technology of EM-types.  The arguments in Theorem \ref{312} are finitary and can be adapted to a direct proof of Corollary \ref{314}, modulo a few applications of compactness.

The author thanks Dana Barto\v{s}ov\'{a}, Christian Rosendal and Stevo Todor\v{c}evi\'{c} for helpful conversations and for suggesting crucial references.  The particular proof written for Prop.~\ref{qfi}(2) is due to a very helpful conversation with John Baldwin, Fred Drueck, and Chris Laskowski.  The author thanks the reviewer for the detailed reading and many helpful comments and suggestions.

\subsection{conventions}\label{11} Much of our model-theoretic notation is standard, see \cite{ho93,ma02} for references.  For $t \in \{0,1\}$, by $\varphi^t$ we mean $\varphi$ if $t=0$, and $\neg \varphi$, if $t =1$. For an $L'$-structure $\I$ and a sublanguage $L^\ast \subseteq L'$, by $\I | L^\ast$ we mean the reduct of $\I$ to $L^\ast$. By $\textrm{qftp}^{L'}(i_1, \ldots, i_{n}; \I)$ we mean the complete quantifier-free $L'$-type of $(i_1, \ldots, i_n)$ in $\I$ (if $L'$ is clear, it is omitted.) The complete quantifier-free type of a substructure of $\I$ is the complete quantifier-free type of a tuple that enumerates some substructure of $\I$. By Diag$(\N)$, we mean the atomic diagram of $\N$.  By age($\I$) we mean the class of all finitely-generated substructures of $\I$ closed under isomorphisms.  In this paper, a complete quantifier-free type is always a type in a finite list of variables.

For a tuple $\ov{a} =  ( a_1, \ldots, a_{m} )$ and a subsequence $\sigma = \la i_1, \ldots, i_{k} \ra$ of $\la 1, \ldots, m \ra$, by $\ov{a} \uphp \sigma$ we mean $( a_{i_1},\ldots, a_{i_{k}} )$. For a  subset $Y \subseteq I$, and a type $\Gamma(\{x_i : i \in I \})$, by $\Gamma |_{\{x_i : i \in Y\}}$ we mean the restriction of $q$ to formulas containing variables in $\{x_i : i \in Y\}$.  If a tuple $\ov{a}$ satisfies a type $\Gamma(\ov{x})$ in a structure $\M$, we write $\ov{a} \vDash_{\M} \Gamma$, where $\M$ is omitted if it is the monster model (see Convention \ref{401}.)

We write $\ov{x}, \ov{a}, \ov{\imath}$ to denote finite tuples, and $\alpha, \beta$ to denote ordinals. The underlying set of a structure $\I$ is given by the unscripted letter, $I$.  For a sequence $\eta : = \la \eta_0, \ldots, \eta_{n-1} \ra$, we denote the length by $\ell(\eta) = n$.  Given a tuple $\ov{a} = (a_1,\ldots,a_n)$, by $(\ov{a})_i$  we mean $a_i$ and by $\bigcup \ov{a}$ we mean $\{ a_i : 1 \leq i \leq n \}$.  We often abbreviate expressions $(a_{i_1}, \ldots, a_{i_n})$ by $\ov{a}_{\ov{\imath}}$.

\section{Basic notions}\label{2}

The definition for $\I$-indexed indiscernible sets was first presented in \cite{sh90}.  We set our notation in the following:

\begin{definition}\label{201}(generalized indiscernible set) Fix an $L'$-structure $\I$ and an $L$-structure $\M$ for some languages $L$ and $L'$.  Let $a_i$ be same-length tuples of parameters from $M$ indexed by the underlying set $I$ of $\I$.

\begin{enumerate}
\item We say that $(a_i : i \in I)$ is an $\I$-\emph{indexed indiscernible (set in $\M$)} if for all $n \geq 1$, for all sequences $i_1,\ldots,i_n$, $j_1,\ldots,j_n$ from $I$,

\vspace{.1in}

\hspace{-.1in} $\textrm{qftp}^{L'}(i_1, \ldots, i_n; \I) = \textrm{qftp}^{L'}(j_1, \ldots, j_n; \I) \Rightarrow$

\vspace{.1in}

\hspace{1.4in} $\textrm{tp}^L(a_{i_1}, \ldots, a_{i_n}; \M) = \textrm{tp}^L(a_{j_1}, \ldots, a_{j_n}; \M)$

\vspace{.1in}

\nin We omit $\M$ where it is clear from context.

\item In the case that the $L'$-structure $\I$ is clear from context, we say that the $\I$-indexed indiscernible $(a_i : i \in I)$ is \textit{$L'$-generalized indiscernible}.

\item Given a sublanguage $L^\ast \subseteq L'$, we say the $L'$-generalized indiscernible set $(a_i : i \in I)$ is \emph{$L^\ast$-generalized indiscernible} if it is an $\I | L^\ast$-indexed indiscernible. 

\item A \emph{generalized indiscernible (set)} is an $I$-indexed set $(a_i : i \in I)$ for some set $I$ that is an $\I$-indexed indiscernible for some choice of structure $\I$ on $I$.
\end{enumerate}

We will always assume that generalized indiscernible sets are \emph{nontrivial}, i.e. that whenever $i \neq j$, $a_i \neq a_j$.
\end{definition}

\begin{notation}  For convenience, $\I$ as in Definition \ref{201} is referred to as the \emph{index model} and $L'$ is the \emph{index language}; $\M$ is referred to as the \emph{target model} and $L$ is  the \emph{target language}.  In this paper, parameters $(a_i : i \in I)$ in $\M$ are always assumed to be tuples such that $\ell(a_i)=\ell(a_j)$ for all $i,j$, and without loss of generality we often assume $\ell(a_i) = 1$.
\end{notation}

\begin{convention}\label{401}
For our purposes, there is no loss in generality to assume we are working not just in a target model $\M$ but in a monster model $\Mm$ of $\textrm{Th}(\M)$. From now on we write $\vDash \varphi$ for $\vDash_{\Mm} \varphi$.  We will reserve $L$ for the language of this model.  Parameters with no identified location come from $\Mm$.
\end{convention}

We define certain technical restrictions on $\I$ that we make in this paper and follow with a proposition.

\begin{definition}\label{211} \hspace{.5in} \begin{enumerate}
\item Say that $\I$ has \emph{quantifier-free types equivalent to quantifier-free formulas (\texttt{qteqf})} if for every complete quantifier-free type $q(\ov{x})$ realized in $\I$, there is a quantifier-free formula $\theta(\ov{x})$ equivalent to $q$ in $\I$, i.e. such that $q(\I) = \theta(\I)$.
\item Say that $\I$ is \texttt{qfi} if, for any complete quantifier free type $q(\ov{x})$ realized in $\I$, there is a quantifier-free formula $\theta_q(\ov{x})$ such that Th($\I$)$_\forall \cup \theta_q(\ov{x}) \vdash q(\ov{x})$.
\end{enumerate}
\end{definition}

\begin{obs} If $\I$ realizes finitely many quantifier-free $n$ types, for each $n$, then it is clear that $\I$ is \texttt{qteqf}.  For example, if $\I$ is a uniformly locally finite $L'$-structure where $L'$ is a finite language, or more specifically, $\I$ is an $L'$-structure where $L'$ is a finite relational language, then $\I$ is \texttt{qteqf}.
\end{obs}

\begin{proposition}\label{qfi}  \hspace{.5in} \begin{enumerate}
\item $\I$ is \texttt{qfi} just in case it has \texttt{qteqf}.
\item In the case that $\I$ is a structure in a finite language and is locally finite, then $\I$ is \texttt{qfi}.
\end{enumerate}
\end{proposition}
\begin{proof} 1.  If $\I$ is \texttt{qfi}, clearly it has \texttt{qteqf} (note that $\theta_q \in q$).  Suppose that $\I$ has \texttt{qteqf}.  Fix a complete quantifier-free type $q(\ov{x})$ realized in $\I$ and say it is equivalent to the quantifier-free $\theta_q$ in $\I$. Then, for all $\psi_\alpha \in q$, $\I \vDash \forall \ov{x} (\theta_q(\ov{x}) \ar \psi_\alpha(\ov{x}) )$.  Thus, Th$(\I)_\forall \vdash \forall \ov{x} (\theta_q(\ov{x}) \ar \psi_\alpha(\ov{x}) )$ and so $\I$ is \texttt{qfi}.

2. This surprisingly helpful observation is surely folklore, but we provide a proof for completeness.  Fix $n$.  We will show that $\I$ has \texttt{qteqf}. By assumption, every $n$-tuple from $\I$ generates a finite substructure of $\I$ and $L'$ is finite.  Thus we may enumerate the finite $L'$-structures up to isomorphism type as $(D_i)_{i < \omega}$ ($\omega$ is not important here.)  Let $\ov{d}_i$ be an enumeration of $D_i$, for each $i$, say that $|D_i| = N(i)$.  For a particular $i$, let $\Phi_{gr}^{\ov{d}_i}$ be a formula in variables $(x_1, \ldots, x_{N(i)})$, satisfied by $\ov{d}_i$ in $\I$, that describes the extensions of the relation symbols on $\ov{d}_i$, the graphs of the function symbols on $\ov{d}_i$, and any equalities or inequalities between constant symbols and the $(\ov{d}_i)_j$.  Clearly such a formula exists as a finite conjunction of literals.

Now given a complete quantifier-free $n$-type realized in $\I$, $q(y_1,\ldots,y_n)$, there must be some $l<\omega$ and some $(x_{i_j} : j \leq n)$ for $i_j \leq N(i)$ such that $q(x_{i_1},\ldots,x_{i_n}) \cup \{\Phi_{gr}^{\ov{d}_l}\}$ is consistent.  But then there are terms $\tau_k = \tau_k(x_{i_1},\ldots,x_{i_n})$ such that 

$q(x_{i_1},\ldots,x_{i_n}) \cup \{\Phi_{gr}^{\ov{d}_l}\} \vdash (\tau_k = x_k)$  

\nin for all $1 \leq k \leq N(i)$.  Let $\sigma_k = \tau_k(x_{i_1},\ldots,x_{i_n};y_1, \ldots, y_n)$ and substitute $\sigma_k$ for $x_k$ in $\Phi_{gr}^{\ov{d}_l}$ to obtain

$\Phi_{gr}^{\ov{d}_l}(x_1,\ldots,x_{N(i)}; \sigma_1(\ov{y}), \ldots, \sigma_{N(i)}(\ov{y}))$.

\nin The latter is a quantifier-free formula equivalent to $q$ in $\I$.  By 1. we are done.
\end{proof}

\begin{remark}\label{402} Note that for a locally finite structure $\I$, $\I$ is \texttt{qfi} just in case for every complete quantifier-free type  $q(\ov{x})$ of a finite substructure of $\I$, there is a quantifier-free formula $\theta_q(\ov{x})$ such that Th($\I$)$_\forall \cup \theta_q(\ov{x}) \vdash q(\ov{x})$.
\end{remark}

The assumption made on index models $\I$ for $\I$-indexed indiscernible in \cite{sh90} is exactly that $\I$ has \texttt{qteqf} (equivalently, \texttt{qfi}.)  The statements of \texttt{qteqf} and \texttt{qfi} offer different perspectives on the same condition and we will use the terms interchangeably.

We define what it means for a generalized indiscernible to inherit the local structure of a set of parameters.  In this definition, the parameters and the indiscernible need not be indexed by the same structure, only by structures in the same language.  In \cite{zi88} the below notion is named \emph{lokal wie}.  The same notion is referred to as \emph{based on} in \cite{sc12, kks11, si12} .  Below we promote a synthesis of the two names:

\begin{definition}[locally based on]\label{31} Fix $\I, \J$ $L'$-structures and a sublanguage $L^\ast \subseteq L'$.  Fix a set of parameters $\In := (a_i : i \in I)$. 
\begin{enumerate}
\item We say the $J$-indexed set $(b_i : i \in J)$ 
is \emph{$L^\ast$-locally based on the $a_i$ ($L^\ast$-locally based on $\In$)} if for any finite set of $L$-formulas, $\Delta$, and for any finite tuple $(t_1, \ldots, t_n)$ from $J$, there exists a tuple $(s_1, \ldots, s_n)$ from $I$ such that 

$\qt^{L^\ast}(\ov{t};\J)=\qt^{L^\ast}(\ov{s};\I)$, 

\nin and

$\textrm{tp}^{\Delta}(\ov{b}_{\ov{t}}; \Mm) = \textrm{tp}^{\Delta}(\ov{a}_{\ov{s}}; \Mm)$.

\vspace{.1in}

\noindent We abbreviate this condition by ``the $b_i$ are $L^\ast$-locally based on the $a_i$.''

\item If the $J$-indexed set $(b_i : i \in J)$ is $L'$-locally based on the $a_i$, we omit mention of $L'$. 
\end{enumerate}
\end{definition}

\begin{obs}\label{38} It is easy to see that the property of one indexed set being \emph{locally based on} another is transitive.  Fix $L'$-structures $\I, \J, \J'$, and parameters $(a_i : i \in I)$, $(b_j : j \in J)$ and $\Wn := (c_k : k \in J')$. Then, if $\Wn$ is locally based on the $b_i$, and the $b_i$ are locally based on the $a_i$, we may conclude that $\Wn$ is locally based on the $a_i$.  In fact, we may further conclude that age($\J'$) $\subseteq$ age($\J$) $\subseteq$ age($\I$) by focusing attention on the complete quantifier-free types of substructures.  
\end{obs}

\begin{definition}\label{226} Fix languages $L^\ast \subseteq L'$.  Given an $L'$-structure $\I$ and an $I$-indexed  set $\In:=(a_i: i \in I)$, define the \emph{($L^\ast$-)EM-type of $\In$} to be:
\begin{eqnarray} \textrm{EMtp}_{L^\ast}(\In)(x_i : i \in I) = \{ \psi(x_{i_1}, \ldots, x_{i_n}) :  \psi \textrm{~from~} L, i_1, \ldots, i_n \textrm{~from~} I, \textrm{and for~} \nonumber \\  \textrm{any~}
 (j_1, \ldots, j_n)  \textrm{~from~} I \textrm{~such that~} \textrm{qftp}^{L^\ast}(j_1,\ldots,j_n; \I)=\textrm{qftp}^{L^\ast}(i_1,\ldots,i_n; \I), \nonumber \\
\vDash \psi(a_{j_1}, \ldots, a_{j_n}) \} \nonumber
\end{eqnarray}

\nin If $L^\ast = L'$, we may omit mention of it.
\end{definition}

\begin{remark}
The specific case of the above definition for $\I$ a linear order is called an ``EM-type'' in \cite{tezi11}.  This notation is not to be confused with EM($I,\Phi$), which in \cite{ba09,sh90} refers to a certain kind of structure.  The relevant similarity is that $\Phi(x_i : i \in I)$ is \emph{proper} for $(\I,\textrm{Th}(\Mm))$ in the sense of \cite{ba09,sh90} if it is the set of formulas satisfied in $\Mm$ by an $\I$-indexed indiscernible.  By Prop.~(\ref{34}) \ref{992}., given an $L'$-structure $\I$, $L'$-EM-types indexed by $I$ may always be extended to a set $\Phi$ proper for $(\I,\textrm{Th}(\Mm))$, provided that $\I$-indexed indiscernible sets have the modeling property.
\end{remark}

The following notation for the type of an indiscernible follows \cite{ma02}.  In the classical case of order indiscernibles, where the index structure is a linear order of the form $(\mathbb{N}, <)$, there is a canonical orientation of the variables in any quantifier-free $n$-type (e.g. $q(x_1, \ldots, x_n)$ where $x_1 < \ldots < x_n$.)  Here we deal with an arbitrary structure $\I$ where there may not be such a canonical orientation, and so we define the type of an indiscernible to include all orientations of variables in all types.  From this perspective, the use of canonical orientations of variables is something of an aesthetic device for special cases.

\begin{definition}\label{555} Given an $\I$-indexed indiscernible set $\In := (a_i : i \in I)$, define:

\begin{enumerate}

\item for any complete quantifier-free type $\eta(v_1,\ldots,v_n)$ realized in $\I$:
\begin{eqnarray}
p^\eta(\In) = \{ \psi(x_1, \ldots, x_n) : \psi \textrm{~from~} L \textrm{~and there exists~} i_1 , \ldots , i_n \textrm{~from~} I   \textrm{~such that~} \nonumber \\  (i_1, \ldots, i_n) \vDash_{\I} \eta 
\textrm{~and~} \vDash \psi(a_{i_1}  \ldots, a_{i_n})  \} \nonumber
\end{eqnarray}
	\item  tp$(\In)$ := $\la p^\eta(\In) : n < \omega, \eta \textrm{~is a~} \textrm{complete quantifier-free~} \right.$ 

\hspace{2.3in} $\left. n\textrm{-type realized in~} \I \ra$
\end{enumerate}
\end{definition}

\begin{obs}\label{400}
Let $\eta(v_1,\ldots,v_n)$ be the complete quantifier free type of a finite substructure of $\I$ in some enumeration.  Suppose there is a permutation $\tau$ of $\{1, \ldots, n\}$ such that realizations of $\eta(v_1, \ldots, v_n)$, $\eta_\tau := \eta(v_{\tau(1)}, \ldots, v_{\tau(n)})$ are isomorphic as tuples.  If $\In$ is an $\I$-indexed indiscernible set, then the following information will be contained in tp($\In$): $\psi(x_1,\ldots,x_n) \in p^\eta(\In) \Leftrightarrow \psi(x_1,\ldots,x_n) \in p^{\eta_\tau}(\In)$.
\end{obs}

\begin{remark} The set $p^{\eta}(\In)$ does not seem terribly useful for a set of parameters $\In$ if $\In$ is not generalized indiscernible, as $p^{\eta}(\In)$ may not be a consistent type.  $\textrm{EMtp}_{L'}(\In)$ is always a consistent type, though may be trivial.
\end{remark}

The following definitions are for Prop \ref{34}.

\begin{definition}\label{33} Fix an $L'$-structure $\I$ and a language $\Ll$.  We define $\textbf{Ind}(\I,\Ll)$ to be
\begin{eqnarray}
\textbf{Ind}(\I,\Ll)(x_i : i \in I) := \{\varphi(x_{i_1},\ldots,x_{i_n}) \ar \varphi(x_{j_1},\ldots,x_{j_n}) : n<\omega, \ov{\imath},\ov{\jmath} \textrm{~from~} I, \nonumber \\ \qt^{L'}(\ov{\imath};\I)=\qt^{L'}(\ov{\jmath};\I), 
\varphi(x_1,\ldots,x_n) \in \Ll\} \nonumber
\end{eqnarray}
\end{definition}

\begin{definition} Let $\Gamma(x_i : i \in I)$ be an $L$-type and $\Uu = (a_i : i \in I)$ an $I$-indexed set of parameters in $\Mm$.  We say that \emph{$\Gamma$ is finitely satisfiable in $\Uu$} if for every finite $I_0 \subseteq I$, there is a $J_0 \subseteq I$, a bijection $f: I_0 \ar J_0$ and an enumeration $\ov{\imath}$ of $I_0$ such that $\qt^{L'}(\ov{\imath}; \I) = \qt^{L'}(f(\ov{\imath});\I)$ and
 $(a_{f(i)} : i \in I_0)$ $\vDash$ $\Gamma|_{\{x_i : i \in I_0\}}$
\end{definition}

\begin{obs}\label{228} If $\I$ and $\J$ are $L'$-structures with the same age, then they realize the same complete quantifier-free types: Suppose $\ov{\imath}$ from $\I$ realizes complete quantifier-free type $\eta(v_1,\ldots,v_n)$.  Since $\I$ and $\J$ have the same age, the substructure of $\I$ generated by $\ov{\imath}$ is isomorphic to some substructure of $\J$.  An isomorphism taking one substructure to the other takes $\ov{\imath}$ to a tuple $\ov{\jmath}$ from $\J$ satisfying the same complete quantifier-free type.
\end{obs}

In the next proposition we detail how two sets of parameters indexed by $L'$-structures may interact by way of EM-type, tp, and the property of being \emph{locally  based on}.  These sets of parameters are indexed by sets $I, J$, and the parameters may or may not be indiscernible according to the intended structures $\I, \J$ on $I, J$. The following  table illustrates the roles of the different bold-face letters:

\vspace{.1in}

\begin{tabular}{|c|c|c|}\hline
indexing set	& $\I/\J\textrm{-indexed indiscernible set}$ & $I/J\textrm{-indexed set}$ \\\hline
$I$			& $\In = (c_i)_i, \Wn = (d_i)_i$					& $\Un = (a_i)_i, \Vn$ \\\hline
$J$			& $\Jn = (b_i)_i$						& $\Tn = (e_i)_i$ \\\hline
\end{tabular}

\vspace{.1in}

\begin{proposition}\label{34} Fix an $L'$-structure $\I$, any $I$-indexed set of parameters 

\noindent $\Un = (a_i : i \in I)$ (possibly indiscernible), and an $\I$-indexed indiscernible set $\In = (c_i : i \in I)$. Let $\J$ be an $L'$-structure with the same age as $\I$ and let $\Jn:= (b_i : i \in J)$ be any $\J$-indexed indiscernible set.  Assume $\I \subseteq \J$ is a substructure in items \ref{992}., \ref{986}., \ref{988}.

\begin{enumerate}
\item\label{991} For any complete quantifier-free type $\eta$ realized in $\J$, if $p^\eta(\In) \subseteq p^\eta(\Jn)$, then $p^\eta(\Jn) \subseteq p^\eta(\In)$
\item\label{993}[two sets of indiscernibles]  $\Jn$ is locally based on the $c_i$ just in case tp$(\In)$=tp$(\Jn)$.
\item\label{992}[two sets of parameters] A $J$-indexed set of parameters $\Tn = (e_i : i \in J)$ is locally based on the $a_i$ just in case EMtp$_{L'}(\Tn) \supseteq$ EMtp$_{L'}(\Un)$.
\item\label{981} For an $I$-indexed set of parameters $\Vn$, $\Vn \vDash$ EMtp$_{L'}(\Un)$ if and only if EMtp$_{L'}(\Vn)$ $\supseteq$ EMtp$_{L'}(\Un)$.
\item\label{982} For an $\I$-indexed indiscernible set $\Wn:=(d_i : i \in I)$, tp$(\Wn)$=tp$(\In)$ just in case $\Wn \vDash$ EMtp$_{L'}(\In)$, just in case EMtp$_{L'}(\Wn)$=EMtp$_{L'}(\In)$.
\item\label{984} If \textbf{Ind}($\I$,$L$) is finitely satisfiable in $\Un$, then there is an $\I$-indexed indiscernible $\Wn:= (d_i : i \in I)$ locally based on the $a_i$.
\item\label{986}  There is a $J$-indexed set of parameters $\Tn =(e_i : i \in J)$ such that EMtp$_{L'}(\Un)$ $\subseteq$ EMtp$_{L'}(\Tn)$.
\item\label{988}   Suppose $\Tn$ is any $J$-indexed set of parameters, and $L^\ast \subseteq L'$.  If EMtp$_{L'}(\Un) \subseteq$ EMtp$_{L'}(\Tn)$, then  EMtp$_{L^\ast}(\Un) \subseteq$ EMtp$_{L^\ast}(\Tn)$.
\end{enumerate}
\end{proposition}

\begin{proof} \begin{enumerate}
\item[\ref{991}.] Suppose $p^\eta(\In) \subseteq p^\eta(\Jn)$.  Let $\varphi(\ov{x}) \in p^\eta(\Jn)$.    Assume, for contradiction, there is no tuple from $I$ witnessing that $\varphi \in p^\eta(\In)$.  Then there is a tuple from $I$ that witnesses that $(\neg \varphi) \in p^\eta(\In)$,  by Obs.~\ref{228} and the fact that $\I$ and $\J$ have the same age. Since $\Jn$ is indiscernible and $\varphi(\ov{x}) \in p^\eta(\Jn)$, in fact for all $\ov{\jmath}$ from $J$ satisfying $\eta$, $\vDash \varphi(\ov{b}_{\ov{\jmath}} )$, and so it is not possible that $(\neg \varphi) \in p^\eta(\Jn)$, as our assumption would have us conclude.
\item[\ref{993}.] Suppose that $\Jn$ is locally based on the $c_i$.  Fix a complete quantifier-free type $\eta(\ov{v})$ realized in $\J$.  By \ref{991}., we need only show that $p^\eta(\In) \subseteq p^\eta(\Jn)$ to show that tp($\In$)=tp$(\Jn)$.  Suppose some tuple from $I$ witnesses that $\varphi(\ov{x}) \in p^\eta(\In)$.  Then by indiscernibility, every tuple $\ov{\imath}$ from $I$ satisfying $\eta$ is witness to $ \vDash \varphi(\ov{c}_{\ov{\imath}})$.  By the property of being \emph{locally based on}, it would be impossible for a tuple $\ov{\jmath}$ from $J$ satisfying $\eta(\ov{v})$ to have $ \vDash\neg \varphi(\ov{b}_{\ov{\jmath}})$.  Thus all tuples $\ov{\jmath}$ from $\J$ satisfying $\eta$ (and there is at least one) witness that $\varphi \in p^\eta(\Jn)$.

The other direction follows from the technique in \ref{992}. for representing $\Delta$-types as formulas.
\item[\ref{992}.] Suppose that $\Tn$ is locally based on the $a_i$ and fix $\varphi(x_{i_1},\ldots,x_{i_n}) \in$ EMtp$_{L'}(\Un$). Let $\ov{\imath} := (i_1,\ldots,i_n)$.  If $ \varphi(x_{i_1},\ldots,x_{i_n}) \notin$ EMtp$_{L'}(\Tn)$, then $ \vDash \neg \varphi(\ov{e}_{\ov{\jmath}})$ for some $\ov{\jmath}$ from $J$ satisfying the same quantifier-free type as $\ov{\imath}$.  By assumption, there exists $\ov{\imath}'$ from $I$ satisfying the same quantifier-free type as $\ov{\jmath}$ and $\vDash \neg \varphi(\ov{a}_{\ov{\imath}'} )$.  But the condition $\varphi(x_{i_1},\ldots,x_{i_n}) \in$ EMtp$_{L'}(\Un$), implies that such an $\ov{\imath}'$ cannot exist.

Suppose EMtp$_{L'}(\Tn)$ $\supseteq$ EMtp$_{L'}(\Un$).  Fix a finite $\Delta \subset L$ and any $\ov{e}_{\ov{\jmath}}$ from $\Tn$.  Let $\ov{\jmath} := (j_1,\ldots,j_n)$. For contradiction, suppose that: 
\begin{eqnarray}\label{eq5} \textrm{~no~} \ov{\imath} \textrm{~exists in~} I, \textrm{~with the same quantifier-free type as~} \ov{\jmath} \nonumber \\
\textrm{~and such that~}  \ov{a}_{\ov{\imath}} \equiv_\Delta \ov{e}_{\ov{\jmath}}
\end{eqnarray}
Let $\varphi$ be the conjunction of positive and negative instances of formulas from $\Delta$ satisfied by $\ov{e}_{\ov{\jmath}}$.  So $\vDash \varphi(\ov{e}_{\ov{\jmath}} )$. By Eq.~(\ref{eq5}), for arbitrary $\ov{\imath} = (i_1, \ldots, i_n)$ from $I$ with the same quantifier-free type as  $\ov{\jmath}$, $\vDash \neg \varphi(\ov{a}_{\ov{\imath}} )$. Thus $\neg \varphi(x_{i_1},\ldots,x_{i_n}) \in$ EMtp$_{L'}(\Un$) $\subseteq$ EMtp$_{L'}(\Tn)$.  But then since $\ov{\jmath}$ satisfies the same quantifier free type as $\ov{\imath}$,  $\vDash \neg \varphi(\ov{e}_{\ov{\jmath}} )$, contradiction.
\item[\ref{981}.] Clear.
\item[\ref{982}.] This follows because the indiscernibility assumption conflates the ``there exists'' condition in tp($\In$) with the ``for all'' condition in EMtp$_{L'}(\In$).
We use \ref{993}. and \ref{992}. to conclude that tp$(\Wn)$=tp$(\In)$ $\Leftrightarrow$ EMtp$_{L'}(\Wn)$ $\supseteq$ EMtp$_{L'}(\In)$.  However, the first condition is symmetric and $\Wn$, $\In$ are both $\I$-indexed indiscernible sets, so we may substitute EMtp$_{L'}(\Wn)$ $=$ EMtp$_{L'}(\In)$ for the second condition. To obtain the equivalence with ``$\Wn \vDash$ EMtp$_{L'}(\In)$'', use \ref{981}.
\item[\ref{984}.]  First observe that if $\Gamma(x_i:i\in I)$ is finitely satisfiable in $\Un$, then 
$\Gamma ~\cup$ EMtp$_{L'}(\Un$) is satisfiable.  So there exists $\Wn$ satisfying \textbf{Ind}($\I,L$) $\cup$ EMtp$_{L'}(\Un$). Thus $\Wn$ is generalized indiscernible and $\Wn \vDash$ EMtp$_{L'}(\Un$).  By \ref{992}. and \ref{981}., $\Wn$ is locally based on the $a_i$.
\item[\ref{986}.] We obtain $\Tn = (e_i : i \in J)$ as a realization of the type 
\begin{eqnarray}
\Gamma(x_j : j \in J) = \{ \varphi(x_{j_1},\ldots,x_{j_n}) : \ov{\jmath} \textrm{~from~} J \textrm{~such that for some~} \ov{\imath} \textrm{~from~} I \textrm{~with~} \nonumber \\
\qt^{L'}(\ov{\jmath};\J)=\qt^{L'}(\ov{\imath};\I), \varphi(x_{i_1},\ldots,x_{i_n}) \in \textrm{~EMtp}_{L'}(\Un) \} \nonumber
\end{eqnarray}
\nin But this type is clearly finitely satisfiable in $\Un$, as $\I$ and $\J$ have the same age.
\item[\ref{988}.] This is clear, as a union of quantifier-free $L'$-types is equivalent to each quantifier-free $L^\ast$-type.  
\end{enumerate}
\end{proof}

\ms

For an $L'$-structure $\I$, if $\I$-indexed indiscernibles have the modeling property, we may find $\J$-indexed indiscernibles locally based on an $\I$-indexed set of parameters, for any $L'$-structure $\J$ with $\textrm{age}(\J) \subseteq \textrm{age}(\I)$, as is observed in \cite{zi88,sc12} (equivalently, if every complete quantifier-free type realized in $\J$ is realized in $\I$.)  We prove a weaker result below, for clarity.  The term ``stretching''  is well-known terminology in the linear order case (see \cite{ho93,ba09}.) 

\begin{definition}\label{227} Fix $L'$-structures $\I$ and $\J$ such that age($\J$)$=$age($\I$).  Given an $\I$-indexed indiscernible $\In = (a_i : i \in I)$, we say a $\J$-indexed indiscernible $\Jn=(b_i : i \in J)$ is a \emph{stretching of $\In$ onto $\J$} if tp($\In$)=tp($\Jn$).
\end{definition}

\nin The lemma below is only a slight generalization of \citep[{chap.~VII, Lemma 2.2}]{sh90} in that the \texttt{qteqf} hypothesis is not needed.  

\begin{lemma}\label{229} For any $L'$-structures $\I$ and $\J$ such that age$(\J)$=age$(\I)$ and $\I$-indexed indiscernible $\In = (a_i : i \in I)$, there is a stretching of $\In$ onto $\J$.
\end{lemma}

\begin{proof} Fix $\In = (a_i : i \in I)$, $\I$, $\J$ as above.  Define $\Gamma$ to be the type:
\begin{eqnarray}\Gamma(x_s : s \in J) := \{ \varphi(x_{s_1},\ldots,x_{s_n}) : (s_1,\ldots,s_n) \textrm{~from~} J, \eta(v_1,\ldots,v_n) \textrm{~is a} \nonumber \\ \textrm{complete quantifier-free~}  \textrm{type in~} \I, 
\textrm{qftp}(s_1,\ldots,s_n; \J)=\eta,  \nonumber \\ \textrm{~and~} \varphi(x_1,\ldots,x_n) \in p^\eta(\In)\} \nonumber 
\end{eqnarray}
\begin{claim}Any realization $\Jn = (b_i : i \in J)$ of $\Gamma$ will be a stretching of $\In$ onto $\J$.\end{claim} 

\begin{proof} Let $\Jn \vDash \Gamma$.   By Obs.~\ref{228}, $\I$ and $\J$ realize the same complete quantifier-free types.  By Prop.~\ref{34} \ref{991}., to see that tp$(\In)$=tp$(\Jn)$ holds we need only show that $p^{\eta}(\In) \subseteq p^{\eta}(\Jn)$, for an arbitrary complete quantifier-free type $\eta$ realized in $\J$.  Note that any formula $\varphi(\ov{x})$ in $p^{\eta}(\In)$ will automatically be in $p^{\eta}(\Jn)$, by definition of $\Gamma$.  A realization of $\Gamma$ is automatically $\J$-indexed indiscernible by the facts that tp$(\In)$=tp$(\Jn)$ and $\I, \J$ realize the same complete quantifier-free types.
\end{proof}
To see that $\Gamma$ is finitely satisfiable in $\Mm$, take a finite subset $\Gamma_0 \subset \Gamma$.  Let $\{j_k : k \leq N\}$ list all the members of $J$ mentioned in any formula in $\Gamma_0$.  Let $B$ be the substructure of $J$ generated by $\{j_k : k \leq N\}$.  By assumption, there is a substructure $A$ of $\I$ isomorphic to $B$, by some isomorphism $f: B \rightarrow A$.  Then $(f(j_k))_{k \leq N}$ has the same complete quantifier free type as $(j_k)_{k \leq n}$ and the tuple $(a_{f(j_k)} : k \leq N)$ works to satisfy $\Gamma_0(x_{j_0},\ldots,x_{j_{N}})$, by generalized indiscernibility of $\In$.
\end{proof}

\section{Modeling property and Ramsey classes}\label{3}

In applications one looks for $\I$-indexed indiscernibles to have the \emph{modeling property}, meaning that $\I$-indexed indiscernible sets can be produced in the monster model of any theory so as to inherit the local structure of an initial $I$-indexed set of parameters.

\begin{definition}\label{32}(modeling property) Fix an $L'$-structure $\I$.  We say that \emph{$\I$-indexed indiscernibles have the modeling property} if given any parameters $(a_i : i \in I)$ in the monster model of some theory, $\Mm$, there exists an $\I$-indexed indiscernible $(b_i : i \in I)$ in $\Mm$ locally based on the $a_i$.
\end{definition}

We repeat defintions for Ramsey classes given in \cite{kpt05,ne05}.  

\begin{definition}\label{30} Define an \emph{$A$-substructure of $C$} to be a substructure $A' \subseteq C$ isomorphic to $A$ where we do not reference a particular enumeration of $A'$. 

We refer to the set of $A$-substructures of $C$ as $C \choose A$.
\end{definition}

\begin{remark} We may think of an $A$-substructure of $C$ as the range of an embedding $e : A \ar C$.  If $A$ has no nontrivial automorphisms, then $A$-substructures may be identified with embeddings of $A$ in $C$.
\end{remark}

\begin{definition}  For an integer $k>0$, by a \emph{$k$-coloring of $C \choose A$} we mean a function $f: {C \choose A} \rightarrow \eta$, where $\eta$ is some set of size $k$ (typically $\eta := \{0, \ldots, k-1\}$.)
\end{definition}

\begin{definition} Fix a class $U$ of $L'$-structures, for some language $L'$.  Let $A, B, C$ be structures in $U$ and $k$ some positive  integer.  
\begin{enumerate}
\item By 
\begin{equation}
C \rightarrow (B)^A_k \nonumber
\end{equation}
we mean that for all $k$-colorings $f$ of $C \choose A$, there is a $B' \subseteq C$, where $B'$ is $L'$-isomorphic to $B$ and the restricted map, $f \uphp {B' \choose A}$, is constant.

\item If, for a particular coloring $f: { C \choose A} \rightarrow k$ we have a $B' \subseteq C$ such that $f \uphp {B' \choose A}$ is constant, we say that \textit{$B'$ is homogeneous for this coloring (homogeneous for $f$)}.
\end{enumerate}
\end{definition}

\begin{definition}\label{44} Let $\Uu$ be a class of finite $L'$-structures, for some language $L'$. $\Uu$ is a \emph{Ramsey class} if for any $A$, $B$ $\in$ $\Uu$ and positive integer $k$, there is a $C$ in $\Uu$ such that $C \rightarrow (B)^A_k$.
\end{definition}

\begin{remark} In the case where $L'$ contains a linear ordering, coloring substructures $A \subseteq C$ is equivalent to coloring embeddings of $A$ into $C$.
It is observed in \cite{ne05} that if we color embeddings, we can never find homogeneous $B \subseteq C$ containing $A$, if $A$ has a nontrivial automorphism and we color the embedded copies of $A$ in $C$ different colors.  If $\I$-indexed indiscernibles have the modeling property, then because of the case of $\M$ a linear order, there cannot exist a finite substructure $A \subset \I$ with a nontrivial automorphism (see Observation \ref{400}.)  The example of this phenomenon with $\I$ an unordered symmetric graph is worked out in \cite{sc12}.
\end{remark}

We want some additional notation for the function symbols case.  For the rest of this section we work with index structures $\I$ that are linearly ordered by some relation, $<$.  By \emph{increasing} we will always mean $<^{\I}$-increasing.

\begin{definition}\label{35} For $\I$ locally finite and linearly ordered by $<$, define $\cl(\cdot)$ on $I$ to take finite tuples $\ov{a}$ in increasing enumeration in $\I$ to the smallest substructure of $\I$ containing $\ov{a}$, also listed in increasing enumeration.
\end{definition}

\begin{remark} In Definition \ref{35}, $\cl(\ov{a})$ is a finite, increasing tuple in $\I$.
\end{remark}

\begin{obs}\label{36} Let $\I$ be as in Definition \ref{35}.
For a finite subset $A \subseteq I$, let $C(A):=\bigcup \cl(\ov{a})$, where $\ov{a}$ lists $A$ in increasing order.  Then $C(\cdot)$ defines a closure property on finite subsets $A, B \subseteq I$: i.e., $A \subseteq C(A)$, $C(C(A)) = C(A)$, and if $A \subseteq B$, then $C(A) \subseteq C(B)$.
\end{obs}

\begin{remark} Our use of $\cl(\cdot)$ in the next theorem and also in Corollary \ref{412} is quite similar to the technique of the strong-subtree envelopes in \citep[6.2]{to10}. 
\end{remark}

The next theorem uses some additional notation.

\begin{definition} Fix a structure $\I$ linearly ordered by a relation $<$.  Fix a finite tuple $\ov{b}$ from $I$ and a finite subset $A \subseteq I$.
\begin{enumerate}
\item By $p_{\ov{b}}(\ov{x})$ we mean the complete quantifier-free type of $\ov{b}$ in $\I$.
\item By $p_A(\ov{x})$ we mean $p_{\ov{a}}(\ov{x})$, where $\ov{a}$ is $A$ listed in increasing enumeration.
\item We say that $\ov{b}$ is an \emph{increasing copy of $A$} if the substructure $B$ of $\I$ on $\bigcup \ov{b}$ is isomorphic to $A$.
\item Fix a finite tuple $\ov{i}$ from $A$ (i.e. $\bigcup \ov{\imath} \subseteq A$) and let $\ov{a}$ list $A$ in $<^{\I}$-increasing order.  We say that \emph{$\ov{i}$ isolates $\tau$ in $A$} if $\ov{a} \uphp \tau = \ov{i}$.
\end{enumerate}
\end{definition}

We give the main theorem.

\begin{theorem}\label{37} Suppose that $\I$ is a \texttt{qfi}, locally finite structure in a language $L'$ with a relation $<$ linearly ordering $I$.  
Then $\I$-indexed indiscernible sets have the modeling property just in case age$(\I)$ is a Ramsey class.
\end{theorem}

\begin{proof} $\Leftarrow$: Here we use the locally finite and ordered hypotheses.  Suppose that age($\I$) is a Ramsey class.  Fix an initial set of parameters $\In:=(a_i : i \in I)$ in  $\mathbb{M}$.  We wish to find $\I$-indexed indiscernible $\Jn:=(b_j : j \in \I)$ locally based on the $a_i$. By Prop \ref{34} \ref{984}., it suffices to show that $\textbf{Ind}(\I,L)$ is finitely satisfiable in $\In$.

Let $\eta$ be a complete quantifier free $n$-type realized by some tuple $\ov{\imath}$ in $\I$.  Let $A$ be the substructure generated by $\ov{\imath}$ in $\I$ (say $A$ has size $N$.)  There is some sequence $\tau$ so that $\ov{\imath}$ isolates $\tau$ in $A$.  Fix this $\tau$ and call it $\sigma_\eta$.  If $\ov{\jmath} \vDash_{\I} \eta$, $\bigcup {\cl}(\ov{\jmath})$ is isomorphic to $A$ by the homomorphism induced by $\ov{\jmath} \mapsto \ov{\imath}$.  If $\ov{b}$ is an increasing copy of $A$, then $\ov{b} \uphp \sigma_\eta \vDash_{\I} \eta$ and ${\cl}(\ov{b} \uphp \sigma_\eta) = \ov{b}$.  Note that for realizations $\ov{\jmath} \vDash_{\I} \eta$, ${\cl}(\ov{\jmath}) \uphp \sigma_\eta = \ov{\jmath}$, thus for $\ov{\jmath}, \ov{\jmath}' \vDash_{\I} \eta$, ${\cl}(\ov{\jmath}) = {\cl}(\ov{\jmath}') \Rightarrow \ov{\jmath} = \ov{\jmath}'$.  So we have shown that $\sigma_\eta$ sets up a correspondence 
\begin{equation}\label{eq6}
\ov{\jmath} \mapsto {\cl}(\ov{\jmath})
\end{equation}
between realizations of $\eta$ in $\I$ and copies of $A$ in $\I$.  

Now let $\Gamma_0 \subseteq \Gamma$ be a finite subset.  $\Gamma_0$ mentions only finitely many formulas $\{\varphi_1,\ldots,\varphi_l\}=:\Delta$.  We may assume that the variables occurring in $\Gamma_0$ are $x_{p_1},\ldots,x_{p_r}$ for some increasing tuple $\ov{p}$ in $\I$.  Let $B:= \bigcup \cl(p_1,\ldots,p_r)$ and let $\ov{p}$ isolate the sequence $\tau_B$ in $B$. Let $\eta_1,\ldots,\eta_s$ be the complete quantifier-free types realized in the set $\{p_1,\ldots,p_r\}$.  It suffices to find a copy $B'$ of $B$ in $\I$ such that 
\begin{equation}\label{eq4} \textrm{for all~} 1 \leq t \leq s, \textrm{~for all realizations~} \ov{\jmath},\ov{\jmath}' \textrm{~of~} \eta_t \textrm{~in~} B', \ov{a}_{\ov{\jmath}} \equiv_{\Delta} \ov{a}_{\ov{\jmath}'}
\end{equation}
\nin since then $\ov{b}' \uphp \tau_B \vDash \Gamma_0$, for $\ov{b}'$ the increasing enumeration of $B'$.

The argument in \citep[{Claim 4.16}]{sc12} shows that we only need to accomplish Eq.~(\ref{eq4}) for one $\eta_t$, as the rest follows by induction.  So fix a complete quantifier-free $n$-type $\eta_t$ realized in $\I$.  For some choice of $\ov{\imath} \vDash_{\I} \eta_t$, let $\bigcup {\cl}(\ov{\imath})=:E$.  Linearly order the finitely many $(\Delta,n)$-types, and suppose there are $K$ of them, for some finite $K$.  Define a $K$-coloring on all copies $E'$ of $E$ in $\I$: $E'$ gets the $k$-th color if its increasing enumeration $\ov{e}'$ has the property that $\ov{e}' \uphp \sigma_{\eta_t} =:\ov{\jmath}$ indexes $\ov{a}_{\ov{\jmath}}$ with the $k$-th $\Delta$-type.   By the assumption of a Ramsey class, there is a copy $B_t$ of $B$ in $\I$ that is homogeneous for this coloring.  Since all copies $E'$ of $E$ in $B_t$ get the same color, by definition of the coloring, there is a $(\Delta,n)$-type $\pi(\ov{x})$, and all $\ov{\jmath} \vDash_{\I} \eta$ such that $\ov{\jmath} = \ov{e}' \uphp \sigma_{\eta_t}$ for $\ov{e}'$ the increasing enumeration of some $E' \cong E$ in $B_t$ are such that $\ov{a}_{\ov{\jmath}} \vDash \pi$.  But every realization of $\eta$ in $B_t$ is such a $\ov{\jmath}$ by Eqn.~(\ref{eq6}) and the fact that $\bigcup \cl(\cdot)$ acts as a closure relation under which $B_t$ is closed.  
\end{proof}

\begin{proof} $\Rightarrow$: Let $\K := \textrm{age}(\I)$.  Suppose that $\I$-indexed indiscernible sets have the modeling property.  We want to show that age($\I$) is a Ramsey class.  We adapt the well-known technique of compactness in partition results to our context:

\begin{claim} Let $\I$ be \texttt{qfi}, locally finite and linearly ordered by one of its relations.   If for all $k < \omega$ and $A,B \in \K$: $I \ar (B)^A_k$, then $\K$ is a Ramsey class.
\end{claim}

\begin{proof} Let $T:=$ Th($\I$), $k, A, B, \I$ as above and suppose $A, B$ have cardinality $n, N$, respectively.  Let $L^+ := L' \cup \{P_0,\ldots,P_{k-1}\}$ and consider the following $L^+$-theory $S$.  For the complete quantifier free types $p_D$ for finite substructures $D \subseteq \I$, substitute a formula equivalent modulo $T_\forall$, using the \texttt{qfi} hypothesis.

\begin{eqnarray} S := T_\forall \cup \textrm{Diag}(\I) \cup \{\forall \ov{x} (p_A(\ov{x}) \ar \bigvee_{i<k} P_i(\ov{x}))\} \cup \nonumber \\ 
\{\neg \exists \ov{x} (P_i(\ov{x}) \wedge P_j(\ov{x})): i \neq j < k\} \cup \nonumber \\ 
\{ \neg \exists \ov{x} (p_B(\ov{x}) \wedge \bigvee_{s<k} \left( \bigwedge_{1 \leq i_1<\ldots<i_n \leq N} (p_A(x_{i_1},\ldots,x_{i_n}) \ar P_s(x_{i_1},\ldots,x_{i_n})) \right) ~) \} \nonumber
\end{eqnarray}

\vspace{.1in} 
\nin If we assume that no $C$ exists in $\K$ such that $C \ar (B)^A_k$, then $S$ is finitely satisfiable, by taking finitely generated substructures of $\I$ and a bad coloring on such a substructure in order to interpret the new predicates, $P_i$. Note that the formulas equivalent to complete quantifier-free types in $\I$ are equivalent to the same types in models of $T_\forall$ (in particular, in substructures of $\I$).  By compactness,  $S$ is satisfied by some structure $\J$ whose restriction to the constants in Diag($\I$) is a structure $\I^\ast$ whose $L'$-reduct is isomorphic to $\I$ by some map $f : I^\ast \ar I$.  There is a coloring by the $P_i^{\J}$ of the $A$-substructures of $\J$ for which there is no copy of $B$ in $\J$ homogeneous for this coloring. If we restrict this coloring to ${\I^\ast \choose A}$, there is still no homogeneous copy of $B$.  By standard methods of reducts and expansions, the map $f$ yields a $k$-coloring of the $A$-substructures of $\I$ for which there is no homogeneous copy of $B$.
\end{proof}

Now fix $\I$ as in the statement of the theorem. The proof continues as in \cite{sc12}; we repeat a shortened proof here for completeness.  At this point the \texttt{qfi} hypothesis is no longer needed.

\begin{claim} Fix $A,B \in \K$ and $k<\omega$. Then $I \ar (B)^A_k$.
\end{claim}

\begin{proof}  Fix a $k$-coloring of the $A$-substructures of $\I$, $g : {\I \choose A} \rightarrow \{1, \ldots, k\}$.  Since $\I$ is linearly ordered, we can understand $g$ as being defined on $n$-tuples $\ov{a} \vDash_{\I} p_A$.  We need to find $B' \subseteq I$ isomorphic to $B$, homogeneous for this coloring.

Let $A$ have size $n$.  Fix a language $L = \{R_1,\ldots,R_k\}$ with $k$ $n$-ary relations and construct an $L$-structure $\M$ as follows:

\begin{enumerate}
\item $|\M| = I$
\item The relation $R_s$, $1 \leq s \leq k$, is interpreted as follows:

For $i_1, \ldots , i_n$ from $|\M|$,  

\hspace{-.4in} $R^{\M}_s(i_1, \ldots, i_n) \Leftrightarrow$

\begin{enumerate}
\item $\ov{\imath} \vDash_{\I} p_A$, and
\item $g( ( i_1, \ldots, i_n ) ) = s$
\end{enumerate}
\end{enumerate}

Let $(a_i : i \in I)$ be the $I$-indexed set in $\M$ such that $a_i = i$.  We work in a monster model $\Mm$ of Th($\M$).  By assumption, we can find an $L'$-generalized indiscernible $(b_j : j \in I)$ in $\Mm$ locally based on the $a_i$. Since $\K$=age($\I$), we may find a copy of $B$ in $\I$, $D'$.  By assumption, $D'$ is a finite structure.  Enumerate $D'$ in $<^{D'}$-increasing order as $( j_k : k \leq N )$.  By the modeling property, for $\Delta :=L$, there is some $i_1, \ldots, i_N$ such that 
\begin{eqnarray}\label{e3}
\textrm{qftp}^{L'}(i_1, \ldots, i_N;\I) = \textrm{qftp}^{L'}(j_1, \ldots, j_N;\I), \textrm{and} \nonumber \\
\textrm{tp}^{\Delta}(b_{j_1}, \ldots, b_{j_N}; \M_1 ) = \textrm{tp}^{\Delta}(a_{i_1}, \ldots, a_{i_N}; \M )
\end{eqnarray}
\begin{claim} $D := (i_k : k \leq N) \subseteq I$ is a copy of $B$ in $I$ that is homogeneous for the coloring, $g$.
\end{claim}

\begin{proof} $D \cong D'$, as $\qt^{L'}(\ov{\imath})=\qt^{L'}(\ov{\jmath})$ and $D,D'$ are structures.  So $D$ is a copy of $B$ and it remains to show that $D$ is homogeneous for the coloring, $g$.  The $b_i$ are generalized indiscernible, so there is some choice of $l_0$ so that for any increasing copy $\ov{c}'$ of $A$ in $D'$, $\vDash R_{l_0}(\ov{c}')$.  We show that all copies of $A$ in $D$ are colored $l_0$ under $g$.

Let $\ov{c}$ be any increasing copy of $A$ in $D$.  There is some sequence $\sigma$ so that $\ov{c}$ isolates $\sigma$ in $\ov{\imath}$.  By the first part of Eq.~(\ref{e3}), for $\ov{c}' := \ov{\jmath} \uphp \sigma$, $\ov{c}'$ is an increasing copy of $A$.  Thus $\vDash R_{l_0}(\ov{c}')$. By the second part of Eq.~(\ref{e3}), $\vDash R_{l_0}(\ov{c})$, i.e., $g(\ov{c}) = l_0$.
\end{proof}

\end{proof}
\end{proof}

\subsection{applications}\label{31}

We make use of $L_i$-generalized indiscernible sets for $i=s,1,2$ where the languages $L_i$ are defined as follows.

\begin{definition}\label{313} \hspace{.5in} \begin{enumerate}
\item We fix languages 

\hspace{-.5in}
\begin{tabular}{lll}
$L_s = \{\unlhd,\wedge,\lx,(P_n)_{n<\omega}\}$, &
$L_1 = \{\unlhd,\wedge,\lx,<_{\textrm{len}}\}$, &
$L_0 = \{\unlhd,\wedge,\lx\}$ \\
\end{tabular}
\item We let $I_s,I_1,I_0$ be the intended interpretations of $L_s,L_1,L_0$, respectively, on $\W$: $\unlhd$ is interpreted as the partial tree-order; $\wedge$ as the meet-function in this order; $\lx$ as the lexicographic ordering on sequences extending the partial tree-order; $P_n$ to hold of $\eta$ just in case $\ell(\eta) = n$; $\eta <_{\textrm{len}} \nu$ to hold just in case $\ell(\eta) < \ell(\nu)$.
\end{enumerate}
\end{definition}

\begin{corollary}\label{314} age$(I_0)$ is a Ramsey class.
\end{corollary}

\begin{proof} 
$I_0$-indexed indiscernible sets have the modeling property by a result from \cite{tats12}. For completeness, an alternate proof of this result is given as Theorem \ref{312}.

It remains to verify the conditions of Theorem \ref{37}.  Since $I_0$ is locally finite in a finite language, $I_0$ is \texttt{qfi} by Prop \ref{qfi}.  Thus by Theorem \ref{37}, age($I_0$) is a Ramsey class.
\end{proof}

\begin{corollary}[\cite{fou99}]\label{315} age$(I_s)$ is a Ramsey class.
\end{corollary}

\begin{proof} In \cite{kks11,tats12}, it was concluded that $I_s$-indexed indiscernible sets have the modeling property, relying on a key result from \cite{sh90}.\footnote{By Theorem \ref{37}, Corollary \ref{412} presents an alternate route to proof.}  It remains to verify the conditions in Theorem \ref{37}.

Note that $I_0 = I_s \uphp \{\unlhd,\wedge,\lx\}$.  In Cor \ref{314} we argue that $I_0$ is \texttt{qfi} by way of Remark \ref{402}.  Let $T_s$ be the theory of $I_s$ and $T_0$ the theory of $I_0$.  

Thus, for any complete quantifier-free $(L_0, m)$-type of a substructure of $I_0$, $p$, there exists an $(L_0, m)$-formula $\theta_p$ such that:
\begin{equation}\label{e1}
(T_0)_\forall \cup \{\theta_p(\ov{x}) \} \vdash p(\ov{x})
\end{equation}
  For any complete quantifier-free $(L_s,m)$-type $q(\ov{x})$ realized in $I_s$, there is some $p_0$ so that $p_0 = q \uphp L_0$.  Thus, for some choice of $t_l \in \{0,1\}$ for $l<\omega$:
\begin{equation}
p_{0}(\ov{x}) \cup \{P_l(x_i)^{t_l} :i<m, l<\omega\} \vdash q(\ov{x})
\end{equation}
Using Eq.~(\ref{e1}) we have,
\begin{equation}
(T_s)_\forall \cup \{\theta_{p_0}(\ov{x})\} \cup \{P_l(x_i)^{t_l} :i<m, l<\omega\} \vdash q(\ov{x})
\end{equation}
We use the facts that, for all $i \neq k < \omega$,
\begin{equation}
(T_s)_\forall \vdash (\forall y \neg (P_i(y) \wedge P_k(y)))
\end{equation}
and any complete quantifier-free type $q$ realized in $I_s$ contains at least one $P_k(x_j)$ for every $j<m$ (though in other models of $T_s$ this may not be the case.)  Thus there exist $i_0, \ldots, i_{m-1} < \omega$ such that, 
\begin{equation}
(T_s)_\forall \cup \{\theta_{p_0}(\ov{x}) \wedge (\bigwedge_{j<m} P_{i_j}(x_j)) \} \vdash q(\ov{x})
\end{equation}

Thus we have shown that $I_s$ is \texttt{qfi}.  By Theorem \ref{37}, age($I_s$) is a Ramsey class.
\end{proof}

We give an additional remark in connection with {\citep[{Example 17}]{tats12}}.  Here the authors provide the example of $I_t := I_0 \uphp \{\unlhd, \lx \}$ and show that $I_t$-indexed indiscernibles do not have the modeling property.  We observe that this fact is also a Corollary of Theorem \ref{37}.  Let $L_t := \{\unlhd, \lx \}$.

\begin{corollary}[\cite{tats12}] $I_t$-indexed indiscernibles do not have the modeling property.
\end{corollary}

\begin{proof} Let $K_t := \textrm{age}(I_t)$.  By Theorem \ref{37}, $I_t$-indexed indiscernibles have the modeling property just in case $K_t$ is a Ramsey class, by a quick verification of the conditions. By {\citep[{Theorem 4.2(i)}]{ne05}} and the presence of a linear ordering, if $K_t$ is a Ramsey class, then $K_t$ has the amalgamation property.  However, an example analyzed in {\citep[{Example 17}]{tats12}} provides the counterexample to amalgamation.  Let $A$ be the finite structure given by $a_0 \unlhd a_1, a_2, a_3$ and $a_0 \lx a_1 \lx a_2 \lx a_3$.  Let $B_i$ be the structures below, where a diagonal edge between nodes denotes that the bottom node is $\unlhd$-related to the top node, the absence of an edge between nodes denotes no $\unlhd$-relation, and $\lx$ both refines $\unlhd$ and obeys the rule that $x \lx y$ if $x$ is to the left of $y$ on the page.  Then $A$ $L_t$-embeds into $B_1, B_2$ by $a_i \mapsto b_i, c_i$.

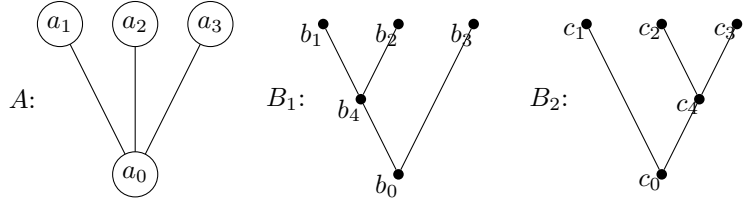
\begin{pgfpicture}{-5cm}{-.5cm}{2cm}{2.5cm}

\pgfputat{\pgfxy(-4.5,1)}{\pgfbox[center,center]{$A$:}}

\pgfnodecircle{Node41}[stroke]{\pgfxy(-3,0)}{.3cm}
\pgfnodecircle{Node31}[stroke]{\pgfxy(-3,2)}{.3cm}
\pgfnodecircle{Node11}[stroke]{\pgfxy(-4,2)}{.3cm}
\pgfnodecircle{Node51}[stroke]{\pgfxy(-2,2)}{.3cm}
\pgfnodeconnline{Node41}{Node31}
\pgfnodeconnline{Node41}{Node11}
\pgfnodeconnline{Node41}{Node51}

\pgfputat{\pgfxy(-3,0)}{\pgfbox[center,center]{$a_0$}}
\pgfputat{\pgfxy(-3,2)}{\pgfbox[center,center]{$a_2$}}
\pgfputat{\pgfxy(-4,2)}{\pgfbox[center,center]{$a_1$}}
\pgfputat{\pgfxy(-2,2)}{\pgfbox[center,center]{$a_3$}}

\pgfputat{\pgfxy(-1,1)}{\pgfbox[center,center]{$B_1$:}}

\pgfnodecircle{Node4}[fill]{\pgfxy(.5,0)}{.07cm}
\pgfnodecircle{Node3}[fill]{\pgfxy(.5,2)}{.07cm}
\pgfnodecircle{Node2}[fill]{\pgfxy(0,1)}{.07cm}
\pgfnodecircle{Node1}[fill]{\pgfxy(-.5,2)}{.07cm}
\pgfnodecircle{Node5}[fill]{\pgfxy(1.5,2)}{.07cm}
\pgfnodeconnline{Node1}{Node2}
\pgfnodeconnline{Node2}{Node3}
\pgfnodeconnline{Node2}{Node4}
\pgfnodeconnline{Node4}{Node5}

\pgfputat{\pgfxy(.5,0)}{\pgfbox[right,top]{$b_0$}}
\pgfputat{\pgfxy(.5,2)}{\pgfbox[right,top]{$b_2$}}
\pgfputat{\pgfxy(-.5,2)}{\pgfbox[right,top]{$b_1$}}
\pgfputat{\pgfxy(1.5,2)}{\pgfbox[right,top]{$b_3$}}
\pgfputat{\pgfxy(0,1)}{\pgfbox[right,top]{$b_4$}}

\pgfputat{\pgfxy(2.5,1)}{\pgfbox[center,center]{$B_2$:}}

\pgfnodecircle{Node20}[fill]{\pgfxy(4,0)}{.07cm}
\pgfnodecircle{Node30}[fill]{\pgfxy(4,2)}{.07cm}
\pgfnodecircle{Node40}[fill]{\pgfxy(4.5,1)}{.07cm}
\pgfnodecircle{Node10}[fill]{\pgfxy(3,2)}{.07cm}
\pgfnodecircle{Node50}[fill]{\pgfxy(5,2)}{.07cm}
\pgfnodeconnline{Node10}{Node20}
\pgfnodeconnline{Node20}{Node40}
\pgfnodeconnline{Node30}{Node40}
\pgfnodeconnline{Node40}{Node50}

\pgfputat{\pgfxy(4,0)}{\pgfbox[right,top]{$c_0$}}
\pgfputat{\pgfxy(4,2)}{\pgfbox[right,top]{$c_2$}}
\pgfputat{\pgfxy(3,2)}{\pgfbox[right,top]{$c_1$}}
\pgfputat{\pgfxy(5,2)}{\pgfbox[right,top]{$c_3$}}
\pgfputat{\pgfxy(4.5,1)}{\pgfbox[right,top]{$c_4$}}
\end{pgfpicture}

Suppose there exists some amalgam $C$ for $(A,B_1,B_2)$.  By a small abuse of notation, we use the labels ``$b_i, c_i$,'' $0 \leq i \leq 4$, to refer to the \textit{images} of these points in $C$.  First, observe that $b_4, c_4$ in $C$ must be $\unlhd$-comparable (by inspection of $K_t$,) as both points are $\unlhd$-predecessors of the same point, $b_2(=c_2)$.  If $b_4 \unlhd c_4$, then $b_4 \unlhd c_4 \unlhd c_3 = b_3$, contradicting the data in $B_1$.  If $c_4 \unlhd b_4$, then $c_4 \unlhd b_4 \unlhd b_1 = c_1$, contradicting the data in $B_2$.  Thus, no such amalgam exists.
\end{proof}

\section{Appendix}\label{4}

As an application of EM-types, we give an alternate proof that $I_0$-indexed indiscernible sets have the modeling property.  This proof eschews \citep[{App. 2.6}]{sh90} in favor of Lemma \ref{411} below, whose statement is taken from \cite{nvt10}, where the original result is attributed to \cite{fou99}.


First we clarify the notion of height we are using.

\begin{definition} Fix a finite tree $T$ partially ordered by $\unlhd$, and let $\nu \in T$.
\begin{enumerate}
\item We say that $\textrm{ht}(\nu) = |\{ \eta : \eta \unlhd \nu, \eta \neq \nu \}|$
\item We say that $\textrm{ht}(T) = \textrm{max} \{\textrm{ht}(\nu) : \nu \in T \}$
\end{enumerate}
\end{definition}

\begin{lemma}[{\citep[{2 (2.2) Lem.~2}]{nvt10}}]\label{411}
Fix $m \in \omega$ and let ${\K}^m_u$ be the class of all finite $L_t$-substructures of $\omega^{\leq m}$ of height $m$, all of whose maximal nodes have height $m$.\footnote{The latter condition is not entirely explicit in the statement, but appears in the proof and is intended by the author.}  Then ${\K}^m_u$ is a Ramsey class.
\end{lemma}

\begin{corollary}\label{412} $\K_s$ is a Ramsey class.
\end{corollary}

\begin{proof} The idea is simple, but we fill in the steps. Fix $D_t$ in ${\K}^m_u$.  We may interpret the $(P_n)_n$ naturally on $D_t$ so that for $\eta \in D_t$, $\textrm{ht}(\eta) = n \leftrightarrow P_n(\eta)$, and we may interpret the meet function $\wedge$ on $D_t$ in the usual way, as it is definable from $\unlhd$.  
In this way we obtain a natural $L_s$-expansion of $D_t$, which we call $\texttt{exp}(D_t)$.  In fact any $L_t$-embedding $f : A_t \ar B_t$ for $A_t, B_t \in {\K}^m_u$ naturally induces an $L_s$-embedding $\bar f : \texttt{exp}(A_t) \ar \texttt{exp}(B_t)$.

Fix $D \in \K_s$ such that $n$ is maximal so that $P^D_n \neq \emptyset$, and let $n \leq m$.  We define an $L_t$-structure from $D$ uniquely up to $L_t$-isomorphism.  Let $k$ be least so that the $L_s$-substructure $E_m \subseteq I_s$ on the set $k^{\leq m}$ contains a copy of $D$, and fix one such copy $D' \subseteq E_m$.  Suppose that $D'$ has $i$-many $\unlhd$-maximal elements, and choose a size-$i$ subset $Y$ of $k^m$ that $\unlhd$-majorizes these maximal elements.
Let $\texttt{fill}_m(D')$ be the $L_t$-reduct in $E_m$ on the set $\{ \eta \in k^{\leq m} : (\exists x \in Y) ~\eta \unlhd x \}$.  Then $\texttt{fill}_m(D') \in \K^m_u$.
There is a first-order $L_t$-formula $\Psi = \Psi_D$ that carves out $D'$, i.e. $\Psi(\texttt{fill}_m(D')) = D'$. For an $L_t$-structure $D_t \cong_{L_t} \texttt{fill}_m(D')$, let $\Ss(D_t)$ be defined as the $L_s$-substructure of $\texttt{exp}(D_t)$ defined on the set $\Psi(D_t)$.  Then $\Ss(D_t) \cong_{L_s} D'$.

Fix $A, B$ in $\K_s$ and $k \in \omega$.  Let $m$ be maximal so that $P_m^B$ is nonempty.
By Lemma \ref{411}, we may choose $C_t \in {\K}^m_u$ so that 
\begin{equation}\label{700}
C_t \ar (\texttt{fill}_m(B))^{\texttt{fill}_m(A)}_k 
\end{equation}
Let $C := \texttt{exp}(C_t)$.  

\begin{claim} $C \ar (B)^A_k$. 
\end{claim}

\begin{proof} Fix a coloring $c : {C \choose A} \ar k$.  
We convert $c$ into a coloring $c': {C_t \choose \texttt{fill}_m(A)} \ar k$ as follows: given $A_t$ a copy of $\texttt{fill}_m(A)$ in $C_t$, let $c'(A_t) := c(\Ss(A_t))$ (by the above, $\Ss(A_t) \cong_{L_s} A$.)
By Eqn.~(\ref{700}), there is a copy $B_t$ of $\texttt{fill}_m(B)$ in $C_t$ homogeneous for this coloring.  Then $\Ss(B_t)$ is a copy of $B$ in $C$ that is homogeneous for $c$, as every copy of $A$ in $\Ss(B_t)$ extends to a copy of $A_t$ in $B_t$.
\end{proof}
\end{proof}

The use of EM-types and Corollary \ref{412} allows us to finitize the proof of Theorem \ref{312} below, up to some applications of compactness.  All the other techniques and ideas below are not new, and may be seen in \cite{sh90,kks11} as well as the original argument in \cite{tats12}.

\begin{theorem}[\cite{tats12}]\label{312} $\I_0$-indexed indiscernible sets have the modeling property
\end{theorem}

\begin{proof} In the following, numbers ``n.''~refer to items from Prop.~\ref{34}.  Let 

\noindent $\In:=(a_i : i \in \W)$ be a set of parameters in a monster model $\Mm$ of some theory.  We must show there is an $I_0$-indexed indiscernible set $L_0$-locally based on the $a_i$. 

\bs

\textbf{step 1.} By Corollary \ref{412} and Theorem \ref{37}, there is an $I_s$-indexed indiscernible $\Tn:=(d_i : i \in \W)$ that is $L_s$-locally based on the $a_i$.   By \ref{992}., EMtp$_{L_s}(\Tn$) $\supseteq$ EMtp$_{L_s}(\In$), so by \ref{988}., 
\begin{equation}\label{e6}
 \textrm{EMtp}_{L_0}(\Tn) \supseteq \textrm{EMtp}_{L_0}(\In)
\end{equation}

\bs

\textbf{step 2.} We aim to find an $I_1$-indexed indiscernible $\Un:=(e_i : i \in \W)$ that is $L_1$-locally based on $\Tn$.  By \ref{984}., $\Un$ may be obtained by the following Claim.

\begin{claim} $\textbf{Ind}(I_1,L)$ is finitely satisfiable in $\Tn$.
\end{claim}

\begin{proof}
Let $F_1 \subset$ \textbf{Ind}($I_1$,$L$) be some finite subset.  There is some $n$ so that all variables occurring in $F_1$ are indexed by nodes in $\omega^{<n}$.  There is some finite set $\Delta \subset \Ll$ such that all formulas occurring in $F_1$ are from $\Delta$.   Let $(\mu^i(x_0,\ldots,x_{m-1}) : i < N)$ enumerate the quantifier-free $L_1$-types of size-$m$ substructures of $\omega^{<n}$, where we may assume $\Delta$ is a set of $L$-formulas in $m$ variables.  Because expansions of $\mu^i$ to complete quantifier-free $L_s$-types may allow $P_k(x_i)$ and $P_k(x_j)$ for $i \neq j$, we do some coding. For any function $f: m \ar m$ and $(j_0, \ldots, j_{m-1}) =: \ov{\jmath} \in \omega^m$ define
\begin{equation}
\mu^i_{f, \ov{\jmath}}:= \mu^i \cup \{P_{j_{f(0)}}(x_0),\ldots,P_{j_{f(m-1)}}(x_{m-1})\}
\end{equation}
By $L_s$-indiscernibility, we know that for any increasing tuple $\ov{\jmath} \in \omega^m$ and $f: m \ar m$, if $\mu^i_{f, \ov{\jmath}}$ is realized in $I_s$, then there is a complete type $p$ in $\Mm$ such that for any $\ov{l} \vDash_{\I_s} \mu^i_{f, \ov{\jmath}}$, ~tp$(\ov{d}_{\ov{l}};\Mm)=p$.  Enumerate the $(\Delta,m)$-types in $\Mm$ as $(\delta_i : 1 \leq i < K)$ for some $K \in \omega$, and fix $\delta_0 := \emptyset$.  Let $N' := N \cdot m^m$.  Fix an enumeration $( (f_\beta, \mu_\beta) : \beta < N')$ of functions $f: m \ar m$ and types $\mu = \mu^i$, for $i<N$.  Let
\begin{equation}
f: [\aleph_0]^m \rightarrow K^{N'} \nonumber
\end{equation}
\nin map an $m$-tuple $\ov{\jmath} \mapsto \alpha$, for $\alpha < K^{N'}$ if 
\begin{enumerate}
\item $(s_\beta)_{\beta < N'}$ is the $\alpha$-th sequence from $K^{N'}$, and
\item for all $\beta < N'$, if there exists $\ov{l}$ from $I_s$ satisfying $\mu^\beta_{f_\beta,\ov{\jmath}}$, then $\textrm{tp}^\Delta(\ov{d}_{\ov{l}};\Mm) = \delta_{s_\beta}$; otherwise, $s_\beta = 0$.
\end{enumerate}
By Ramsey's theorem, there is an infinite subset of $\aleph_0$ that is homogeneous for this coloring.  The $L_1$-subtree of $I_1$ obtained by restricting to the levels in this infinite set indexes a subset of $\Tn= (d_i : i < \W)$, a finite subset of which will satisfy $F_1$.  
\end{proof}

By \ref{992}., EMtp$_{L_1}(\Un)$ $\supseteq$ EMtp$_{L_1}(\Tn)$. Thus,

\begin{equation}\label{e2}
\textrm{EMtp}_{L_0}(\Un) \supseteq_{\textrm{by \ref{988}.}} \textrm{EMtp}_{L_0}(\Tn) \supseteq_{\textrm{by Eq.~(\ref{e6})}} \textrm{EMtp}_{L_0}(\In)
\end{equation}

\textbf{step 3.} If we show that \textbf{Ind}($I_0$,$L$) is finitely satisfiable in $\Un$, then by \ref{984}., there is an $I_0$-indexed indiscernible $\Jn:=(b_i : i \in \W)$ locally based on the $e_i$.  By Eqn.~(\ref{e2}), and \ref{992}., the $e_i$ are $L_0$-locally based on the $a_i$, so by Obs.~\ref{38}, we are done.  It remains to show the following.

\begin{claim}
\textbf{Ind}($I_0$,$L$) is finitely satisfiable in $\Un$. 
\end{claim}

\begin{proof}
A finite subset $F_0 \subset$ \textbf{Ind}($I_0$,$L$) contains only variables indexed by nodes in $\omega^{\leq n}$ for some $n$.  To satisfy $F_0$ in $\Un$, it suffices to show that the type of an $L_0$-generalized indiscernible $k$-branching tree of height $n$ is satisfiable in $\Un$.

We follow \cite{dzsh04} to show that there is an $L_0$-embedding of $\sigma: k^{\leq n} \ar \omega^{<\omega}$ such that for all $i \lx j$, we have $\sigma(i) <_{\textrm{len}} \sigma(j)$.  We define $l_m < \omega, h_m : k^{\leq m} \ar \omega^{<\omega}$ by induction on $m$:
\begin{eqnarray}
h_i(\la \ra) = \la \ra, \textrm{~for all~} i<\omega \\\nonumber
l_m = \textrm{max}\{ \ell(h_m(\eta))+1 : \eta \in k^{\leq m} \} \\\nonumber
h_{m+1}(\la t \ra^{\smallfrown}\nu) = \la t \ra^{\smallfrown}{\underbrace{\la 0, \ldots 0 \ra}_{(t+1)\cdot l_m}}^{\smallfrown}h_m(\nu) 
\end{eqnarray}
Define $\sigma := h_n$.  The range of $k^{\leq n}$ under $\sigma$ is an $L_1$-subtree $W \subset I_1$, sometimes called a ``skew subtree.''  $\Un$ is already $L_1$-generalized indiscernible.  Since  the $L_0$-type of a tuple in $W$ determines its $L_1$-type in $I_1$, $(e_{\sigma(i)} : i \in k^{\leq n})$ is $L_0$-generalized indiscernible.
\end{proof}

\end{proof}

\bibliographystyle{plainnat}
\bibliography{preref}

\end{document}